\documentclass[12pt,psfig, reqno]{amsart}
\usepackage{amsmath, amsthm, amsfonts, amssymb, latexsym, graphicx}
\usepackage[latin1]{inputenc}

\newtheorem{theorem}{Theorem}[section]
\newtheorem{lemma}[theorem]{Lemma}

\theoremstyle{definition}

\newtheorem{remark}[theorem]{Remark}
\numberwithin{equation}{section}

\newcommand{\mC}{\ensuremath{\mathbb{C}}}

 \allowdisplaybreaks
\begin{document}

\title{A value distribution result related to Hayman's alternative}
\author[K. S. Charak]{Kuldeep Singh Charak}
\address{
\begin{tabular}{lll}
&Kuldeep Singh Charak\\
&Department of Mathematics\\
&University of Jammu\\
&Jammu-180 006\\ 
&India\\
\end{tabular}}
\email{kscharak7@rediffmail.com}

\author[A. Singh]{Anil Singh}
\address{
\begin{tabular}{lll}
&Anil Singh\\
&Department of Mathematics\\
&University of Jammu\\
&Jammu-180 006\\
&India
\end{tabular}}
\email{anilmanhasfeb90@gmail.com }

\begin{abstract}
Motivated by Bloch's Principle, we prove a value distribution result for meromorphic functions which is related to Hayman's alternative in certain sense.
\end{abstract}

\renewcommand{\thefootnote}{\fnsymbol{footnote}}

\maketitle

 \section{Introduction and Main Result}
The reader is assumed to be familiar with the standard notations of Nevanlinna  value distribution theory of meromorphic functions (one may refer to \cite{ cherry, Hayman-1}) such as $T\left(r,f\right)$, $m\left(r,f\right)$, $N\left(r,f\right),$ etc. We shall denote the class of all meromorphic functions on a domain $D$ in $\mC$ by $\mathcal{M}(D)$ and we shall write, `$\left\langle f,D\right\rangle\in\mathcal{P}$' for 
`$f\in \mathcal{M}(D)$ satisfies the property $\mathcal{P}$ on $D$' .

   We say that $\phi \in \mathcal{M}(\mC)$ is a small function of  $f \in \mathcal{M}(\mC)$ if $T(r,\phi)=S(r,f)$ as $r\to \infty$ possibly outside a set of $r$ of finite linear measure.

\medskip

    W.K. Hayman  proved the following `Picard type' theorem, also known as Hayman's alternative: 
 
 \begin{theorem}\label{alternative} \cite{Hayman-2} Let $f\in \mathcal{M}(\mC)$ and let $l\geq 1.$ Suppose that $f(z)\neq 0, \mbox{ and } f^{(l)}(z)-1\neq 0$ for all $z\in\mathbb{C}.$  Then $f$ is constant.
\end{theorem}

\medskip

A subfamily $\mathcal{F}$ of $\mathcal{M}(D)$ is said to be normal in $D$ if every sequence of members of $\mathcal{F}$ contains a subsequence that converges locally uniformly (w.r.t. the spherical metric) in $D$. Recall  Bloch's Principle (see\cite{rubel,Schiff}): {\it  A subfamily $\mathcal{F}$ of $\mathcal{M}(D)$ with  $\left\langle f,D\right\rangle\in\mathcal{P}$ for each $f\in \mathcal{F}$ is likely to be normal on $D$ if $\mathcal{P}$ reduces every $f\in \mathcal{M}(\mC)$ to a constant}. Neither Bloch's Principle nor its converse is true (see \cite{charak1,charak2,charak3,lahiri,rubel}).

    According to Bloch's Principle, to every `Picard type' theorem  there corresponds a normality criterion. A normality criterion corresponding to  Theorem \ref{alternative} was proved by Y.Gu  as follows:

 \begin{theorem}\label{Gu} \cite{Gu} Let $\mathcal{F} \subseteq  \mathcal{M}(D)$  and let $l\geq 1.$ Suppose that $f(z)\neq 0, \mbox{ and }f^{(l)}(z)-1\neq 0$ for all $z\in D$ and $f\in\mathcal{F}.$  Then $\mathcal{F}$ is normal in $D.$ 
\end{theorem}
  
	The constants $0 \mbox{ and }1$ in Theorem \ref{alternative} and Theorem \ref{Gu} can be replaced by arbitrary constants $a \mbox{ and }b\neq 0$:
\begin{theorem} \cite{Hayman-2}\label{alternative1} Let $f \in  \mathcal{M}(D)$  and let $l\geq 1.$ Suppose that $f(z)\neq a, \mbox{ and }f^{(l)}(z)-b\neq 0$ for all $z\in\mathbb{C},$ where $a,b\in\mathbb{C},~b\neq 0.$ Then $f$ is constant.
\end{theorem}
\begin{theorem}\cite{Gu}\label{Gu1} Let $\mathcal{F} \subseteq  \mathcal{M}(D)$  and let $l\geq 1.$ Suppose that $f(z)\neq a, \mbox{ and }f^{(l)}(z)-b\neq 0$ for all $z\in D,$  $f\in\mathcal{F}$, where  $a,b\in\mathbb{C}$ with $~b\neq 0.$ Then $\mathcal{F}$ is normal in $D.$ 
\end{theorem}
Note that if $l\geq 1$ and $b\in\mathbb{C}\setminus\left\{0\right\}$, then there is a polynomial $P(z)$ such that $P^{(l)}(z)=b.$ Using this observation, Theorem	\ref{alternative1} and Theorem \ref{Gu1} can be restated as:
\begin{theorem}\label{alternative2} Suppose that $P(z)$ is a polynomial of degree  $l\geq 1$ and $a\in\mathbb{C}.$  If $f\in \mathcal{M}(\mC)$ is such that  $f(z)\neq a \mbox{ and }\left(f(z)-P(z)\right)^{(l)}\neq 0,$   then $f$ is constant.
\end{theorem}
\begin{theorem}\label{Gu2}Suppose that $P(z)$ is a polynomial of degree $l\geq 1$ and $a\in\mathbb{C}.$ If $\mathcal{F} \subseteq  \mathcal M(D)$ is such that each $f\in\mathcal{F}$ satisfies: $$f(z)\neq a\mbox{ and } \left(f(z)-P(z)\right)^{(l)}\neq 0,$$  then $\mathcal{F}$ is normal in $D.$
\end{theorem}
\begin{remark}\label{remark} {\it Put $g=f-P\mbox{ and } R=Q-P$, where $P\mbox{ and }Q$ are polynomials with $deg(P-Q)=deg(Q)=l$ and $Q$ is non-constant. If $f(z)-P(z)\neq 0$ and $\left(f(z)-Q(z)\right)^{(l)}\neq 0,$ then by using Theorem \ref{alternative2}, we find that $f(z)=P(z)+c$, for some constant $c\neq 0.$}
\end{remark}
Remark \ref{remark} shows that Theorem \ref{alternative1} does not hold if $a$ is replaced by some non-constant function.
\begin{remark}
{\it Suppose  $f$ is an entire function such that $f-g$ has only finitely many zeros in the plane, where $g$ is some non-constant entire function. Further, let $$F(z)=\sum_{k=1}^{n}a_k(z)\left(f-g\right)^{(k)}$$ omits $1,$ where $a_k(z)$ are small functions of $f.$ Then by using Theorem 3.2 in \cite{Hayman-1}, we find that  
 $f(z)=g(z)+p(z)$, for some polynomial $p(z).$ Indeed,  
$$T(r,f-g)< \overline{N}(r,f-g)+N(r, {1\over f-g})+\overline{N}\left(r,{1\over F-1}\right)-N_0\left(r,{1\over F'}\right)+S(r,f),$$ where $N_0\left(r,{1\over F'}\right)$ is the counting function of the zeros of $F'$ which are not zeros of $F-1.$

Since $f-g$ is entire and has only finitely many zeros, it follows that
 \begin{flalign*}
T(r,f-g)< \overline{N}\left(r,{1\over F-1}\right)-N_0\left(r,{1\over F'}\right)+S(r,f)\\
\Rightarrow T(r,f-g)< \overline{N}\left(r,{1\over F-1}\right)+S(r,f).
\end{flalign*}
If $f-g$ is transcendental, then $F(z)=1$ must have infinitely many roots, which is a contradiction and hence $f-g$ must be a polynomial, say $p(z);$ that is,  $f(z)=g(z)+p(z).$ }
\end{remark}
  Let $P\mbox{ and }Q$  be polynomials with $1\leq deg(P)<deg(Q)=l$ and $\mathcal{P}$ be the property defined as follows:
	$``\left\langle f,D\right\rangle\in\mathcal{P}\Leftrightarrow f-P\neq 0\mbox{ and }\left(f-Q\right)^{(l)}\neq 0 \mbox{ on }D$''. That is, $f$ satisfies the property $\mathcal{P}$ if, and only if $f-P\mbox{ and }\left(f-Q\right)^{(l)}$ have no zeros in $D.$ With this $\mathcal{P}$, Theorem \ref{Gu2} immediately yields:
	\begin{theorem}\label{Gu3} The family $\mathcal{F}:=\left\{f\in\mathcal{M}\left(D\right):\left\langle f,D\right\rangle\in\mathcal{P}\right\}$ is normal in $D.$
	\end{theorem}
Note that Remark \ref{remark} and Theorem \ref{Gu3} provide a counterexample to the converse of Bloch's Principle.
	
	\medskip

		     W. Schwick   generalized Theorem \ref{Gu} :
	\begin{theorem}\label{schwick} \cite{schiwick} Let $g\not\equiv 0$ be in $\mathcal{M}(D)$ and let $l\in\mathbb{N}.$ Let $\mathcal{F} \subseteq  \mathcal{M}(D)$ be such that $f\neq 0,$  $f^{(l)}\neq g,$ and $f\mbox{ and }g$ have no common poles  for each $f\in\mathcal{F}.$ Then $\mathcal{F}$ is  normal in $D.$
\end{theorem}
	
	According to the converse of Bloch's Principle, one may find a `Picard type' theorem corresponding to Theorem \ref{schwick}, and this is the purpose of this paper. In fact,  we prove the following value distribution result  corresponding to Theorem \ref{schwick} which is related to Hayman's alternative in certain sense:

\begin{theorem}\label{MTA} Suppose that $f\in \mathcal{M}(\mC)$ is transcendental and $\phi$ is a small function of $f$ such that $f\mbox{ and }\phi$ have no common poles. Let $l\in\mathbb{N}$ and $\psi(z)=f^{(l)}(z)$. If $f(z)\neq 0\mbox{ and } \psi(z)\neq \phi(z)$ for all $z\in\mathbb{C},$   then 
 $\psi'(z)= \phi(z)$ and $\psi'(z)= \phi'(z)$ have infinitely many solutions.

\end{theorem}


\section{ Proof of Theorem \ref{MTA}}\label{maintheorem}
Since the proof of Theorem \ref{MTA} is based on Milloux techniques (see \cite{Hayman-1} p.60), we need to prove some key lemmas for the proof of Theorem \ref{MTA}. Throughout this paper, we shall denote $f^{(l)}(z)$ by $\psi(z)$, where $l\in\mathbb{N}$.

\begin{lemma}\label{theorem 3.2} Let $f \in \mathcal{M}(\mC)$  and let $\phi$ be a small function of $f.$ Then for $r\to\infty$ outside a set of finite linear measure,
\begin{equation}\label{2.1}T\left(r,f\right)\leq\bar{N}\left(r,f\right)+N\left(r,\frac{1}{f}\right)+\bar{N}\left(r,\frac{1}{\psi-\phi}\right)-N^0_2\left(r,\psi\right)+S(r,f),
\end{equation}
where $N^0_2\left(r,\psi\right)=N\left(r,\frac{1}{\psi'}\right)-N_0\left(r,\frac{1}{\psi'}\right)$ and  $N_0\left(r,\frac{1}{\psi'}\right)$ is the counting function of zeros of $\psi'$ which are not zeros of $\psi$.
\end{lemma}
\begin{proof} By the second fundamental theorem of Nevanlinna for three small functions (see \cite{Hayman-1} Theorem. 2.5, also see\cite{cherry} Theorem. 5.9.1) with $a_1=0,a_2=\infty\mbox{ and } a_3=\phi$, we have 
\begin{equation}\label{one}
T(r,f)\leq \bar{N}\left(r,\frac{1}{f}\right)+\bar{N}\left(r,{f}\right)+\bar{N}\left(r,\frac{1}{f-\phi}\right)+S(r,f).
\end{equation} as $r\to\infty$ outside a set of $r$ of finite linear measure.

\medskip

Since

\medskip

$\bar{N}\left(r,\frac{1}{f}\right)={N}\left(r,\frac{1}{f}\right)-{N}\left(r,\frac{1}{f'}\right)+{N}_0\left(r,\frac{1}{f'}\right),$\\
$~~~~\bar{N}\left(r,\frac{1}{f-\phi}\right)={N}\left(r,\frac{1}{f-\phi}\right)-{N}\left(r,\frac{1}{f'-\phi '}\right)+{N}_0\left(r,\frac{1}{f'-\phi '}\right)$\\

and

\medskip
 
$\bar{N}\left(r,f\right)=N\left(r,f'\right)-N\left(r,f\right),$\\
therefore (\ref{one}) yeilds (after adding $m\left(r,\frac{1}{f}\right)+m\left(r,\frac{1}{f-\phi}\right)$ to both sides) 
\begin{flalign*} T\left(r,f\right)+m\left(r,\frac{1}{f}\right)+m\left(r,\frac{1}{f-\phi}\right) & \leq  m\left(r,\frac{1}{f}\right)+m\left(r,\frac{1}{f-\phi}\right)\\
&+N\left(r,\frac{1}{f}\right)+N\left(r,\frac{1}{f-\phi}\right) +N\left(r,{f'}\right)-N\left(r,{f}\right) &&\\
&-N\left(r,\frac{1}{f'}\right)+N_0\left(r,\frac{1}{f'}\right) +N_0\left(r,\frac{1}{f'-\phi '}\right)&&\\
&-N\left(r,\frac{1}{f'-\phi '}\right)+S(r,f) &&
\end{flalign*}

 \begin{flalign}\label{eq:2MT}
\Rightarrow m\left(r,\frac{1}{f}\right)+m\left(r,\frac{1}{f-\phi}\right)+m\left(r,{f}\right) 
& \leq  2T\left(r,f\right)-N_1\left(r,f\right) +N_0\left(r,\frac{1}{f'}\right) \nonumber\\
 & +N_0\left(r,\frac{1}{f'-\phi '}\right)-N\left(r,\frac{1}{f'-\phi '}\right)+S(r,f),&&
\end{flalign}

where  $N_1\left(r,f\right)=N\left(r,\frac{1}{f'}\right)+2N\left(r,f\right)-N\left(r,f'\right).$ 

\medskip

Applying (\ref{eq:2MT}) to $\psi=f^{(l)} \mbox{ and put }g=\psi-\phi, $ we have 
\begin{flalign}\label{eq:3MT}
m\left(r,\frac{1}{\psi}\right)+m\left(r,\frac{1}{g}\right)+m\left(r,{\psi}\right) &\leq 2T\left(r,\psi\right)-N_1\left(r,\psi\right)+N_0\left(r,\frac{1}{\psi'}\right) \nonumber\\
	&+N_0\left(r,\frac{1}{g '}\right)-N\left(r,\frac{1}{g'}\right)+S(r,\psi) &&
\end{flalign} as $r\to\infty$ outside a set of $r$ of finite linear measure.\\

   Since 
	
	\medskip
	
	$N\left(r,\psi'\right)-N\left(r,\psi\right)=\bar{N}\left(r,f\right)$\\
	
	\medskip
	
	and 
	
	\medskip
	
	$N\left(r,\frac{1}{g}\right)-N\left(r,\frac{1}{g'}\right)=\bar{N}\left(r,\frac{1}{g}\right)-N_0\left(r,\frac{1}{g'}\right)$
	
	\medskip
	
	and using the first fundamental theorem of Nevanlinna, we have  
\begin{flalign*}\label{eq:4MT} 
2T\left(r,\psi\right)-N_1\left(r,\psi\right)& =m\left(r,{\psi}\right)+m\left(r,\frac{1}{g}\right)+N\left(r,\psi\right)+N\left(r,\frac{1}{g}\right)\\
&-N\left(r,\frac{1}{\psi '}\right)-2N\left(r,\psi\right)+N\left(r,\psi '\right) +S(r,f) &&\\
&=m\left(r,{\psi}\right)+m\left(r,\frac{1}{g}\right) + N\left(r,\frac{1}{g}\right)-N\left(r,\frac{1}{\psi'}\right) &&\\
&+\bar{N}\left(r,f\right)+S(r,f)&&
\end{flalign*}
and hence (\ref{eq:3MT}) reduces to 
\begin{equation}\label{eq:proximity}
	m\left(r,\frac{1}{\psi}\right)\leq \bar{N}\left(r,\frac{1}{g}\right)-N\left(r,\frac{1}{\psi'}\right)+\bar{N}\left(r,f\right)+N_0\left(r,\frac{1}{\psi'}\right)+S(r,f).
	\end{equation}


Also
	\begin{flalign}\label{eq:10MT}
	T\left(r,f\right) &=T\left(r,\frac{1}{f}\right)+O(1)\nonumber\\
	&=m\left(r,\frac{\psi}{f\psi}\right)+N\left(r,\frac{1}{f}\right)+O(1)&&\nonumber\\
	&\leq m\left(r,\frac{\psi}{f}\right)+m\left(r,\frac{1}{\psi}\right)+N\left(r,\frac{1}{f}\right)+O(1)&&\nonumber\\
	&=m\left(r,\frac{1}{\psi}\right)+N\left(r,\frac{1}{f}\right)+S(r,f).&&
	\end{flalign}
Now by using (\ref{eq:proximity}) in  (\ref{eq:10MT}), we get
\begin{flalign}\label{eq:11MT}
T\left(r,f\right) &\leq \bar{N}\left(r,f\right)+N\left(r,\frac{1}{f} \right) +\bar{N}\left(r,\frac{1}{g}\right) -N\left(r,\frac{1}{\psi'}\right)+N_0\left(r,\frac{1}{\psi'}\right)+S(r,f) \nonumber\\
&=\bar{N}\left(r,f\right)+N\left(r,\frac{1}{f}\right)+\bar{N}\left(r,\frac{1}{\psi-\phi}\right)-N^0_2\left(r,{\psi}\right)+S(r,f)&&
\end{flalign}
as $r\to\infty$ outside a set of $r$ of finite linear measure,	
where  $$N^0_2\left(r,\psi\right)=N\left(r,\frac{1}{\psi'}\right)-N_0\left(r,\frac{1}{\psi'}\right)=N\left(r,\frac{1}{\psi}\right)-\bar{N}\left(r,\frac{1}{\psi}\right)$$ counts only repeated zeros of $\psi$ with multiplicity reduced by $1$.   
\end{proof}
\begin{lemma}\label{lemma3.1A}
 Let $f\in\mathcal{M}\left(\mathbb{C}\right)$  and let $\phi$ be a small function of $f$ such that $f \mbox{ and }\phi$ have no common poles.  Then
\begin{equation} lN_1\left(r,f\right)\leq \bar{N_2}\left(r,f\right)+\bar{N}\left(r,\frac{1}{\psi-\phi}\right)+\bar{N}\left(r,\frac{1}{\psi'-\phi}\right)+S(r,f),
\end{equation}
 where $N_1\left(r,f\right)$ is the counting function of simple poles of $f$and $\bar{N_2}\left(r,f\right)$ is the counting function of multiple poles of $f$ counted once.
\end{lemma}
The proof of Lemma \ref{lemma3.1A} is carried out by following the proof of Lemma 3.1 in \cite{Hayman-1} with certain modifications.

\medskip

{\bf Proof of Lemma \ref{lemma3.1A}.}
 Put 
$$G(z)={\left\{\psi'(z)-\phi(z)\right\}^{l+1}\over {\left\{ \phi(z)-\psi(z)\right\}^{l+2}}}.$$
 If $z_0$ is a simple pole  of $f(z),$ then near $z_0$ we have  
$$f(z)={a\over z-z_0}+O(1),$$
and   
$$\psi(z)={(-1)^{l}a (l!) \over \left(z-z_0\right)^{l+1}}+O(1)$$
 near $z_0$, and hence 
$$ \psi'={(-1)^{l+1}a(l+1)! \over \left(z-z_0\right)^{l+2}}\left\{1+O(z-z_0)^{l+2}\right\}.$$
  Since $f$ and $\phi$ have no common poles, therefore near $z_0,$ we have 
$$G(z)={(-1)^{l+1}(l+1)^{l+1} \over {al!}}\left\{1+O(z-z_0)^{l+1}\right\}$$ 
  which implies that $ G(z_0)\neq 0,\infty,$ and  $G'(z)$ has a zero of order atleast $l$ at $z_0$ and so 
\begin{eqnarray}\label{eq:1L}
l N_1\left(r,f\right)\leq N_0\left(r,\frac{1}{G'}\right).
\end{eqnarray}

 Applying Jensen's formula to ${G'/{G}}$ we get
\begin{equation}\label{eq:2L}
N\left(r,\frac{G}{G'}\right)-N\left(r,\frac{G'}{G}\right)=m\left(r,\frac{G'}{G}\right)-m\left(r,\frac{G}{G'}\right)+O(1).
\end{equation}  
Since the only zeros of ${G'}/{G}$ are the zeros of $G'$ which are not zeros of $G,$ we have  
\begin{eqnarray}\label{eq:3L}
N\left(r,\frac{G}{G'}\right)=N_0\left(r,\frac{1}{G'}\right).
\end{eqnarray}
 Also, ${G'}/{G}$ has only simple poles at the zeros and poles of $G,$ so 
\begin{eqnarray}\label{eq:4L}
 N\left(r,\frac{G'}{G}\right)=\bar{N}\left(r, G\right)+\bar{N}\left(r,\frac{1}{G}\right).
\end{eqnarray}

Using (\ref{eq:3L}) and (\ref{eq:4L}) in  (\ref{eq:2L}), we obtain
\begin{flalign}\label{eq:5L}
N_0\left(r,\frac{1}{G'}\right)-\bar{N}\left(r,\frac{1}{G}\right)-\bar{N}\left(r,{G}\right)
=m\left(r,\frac{G'}{G}\right)-m\left(r,\frac{G}{G'}\right)+O(1).
\end{flalign}
Now, from (\ref{eq:1L}), (\ref{eq:2L}) and (\ref{eq:5L}), we have
\begin{flalign}\label{eq:simple pole}
  lN_1\left(r,f\right)  &\leq  N_0\left(r,\frac{1}{G'}\right)\nonumber\\
&= \bar{N}\left(r,\frac{1}{G}\right)+\bar{N}\left(r,{G}\right)+m\left(r,\frac{G'}{G}\right)-m\left(r,\frac{G}{G'}\right)+O(1)&&\nonumber\\
&\leq \bar{N}\left(r,\frac{1}{G}\right)+\bar{N}\left(r,{G}\right)+m\left(r,\frac{G'}{G}\right)+O(1).&& 
\end{flalign}

   Let $z_0$ be a pole of $\psi(z)-\phi(z)$ of order $m$, say. Then near $z_0$ $$\psi(z)-\phi(z)={s_0(z)\over\left(z-z_0\right)^m}$$   for some  function $s_0(z)$ analytic in a neighborhood of $z_0 \mbox{ such that } s_0(z_0)\neq 0.$

\medskip

   Now, there are two cases:
	\begin{itemize}
	\item[Case 1:]  $z_0$ is a pole of $f(z)$. Then  $m=k+l $, where $k>1$ is the multiplicity of $z_0$ as a pole of $f$. Since $z_0$ is not a pole of $\phi,$ we see that $z_0$ is a pole of $\psi'(z)-\phi(z)$ of multiplicity  $m+1.$ Therefore, near $z_0$ 
	$$G(z)=t_0(z)\left(z-z_0\right)^{k-1},$$
	for some  function $t_0(z)$ analytic in a neighborhood of $z_0 \mbox{ such that } t_0(z_0)\neq 0.$ So $z_0$ is a zero of order $k-1$ of $G(z)$.
	
	\item[Case 2:] $z_0$ is a pole of $\phi(z).$ Then $z_0$ is also a pole of $\psi'(z)-\phi(z)$ of multiplicity $m.$ Therefore, near $z_0$ 
	$$G(z)=t_1(z)\left(z-z_0\right)^{m},$$ 
	for some  function $t_1(z)$ analytic in a neighborhood of $z_0 \mbox{ such that } t_1(z_0)\neq 0.$ This shows that $z_0$ is a pole of $G(z)$ of the same multiplicity as that of $\phi(z).$
\end{itemize}
Similarly, looking at the poles of $\psi'(z)-\phi(z),$ we obtain the same conclusion as in the case of poles of $\psi(z)-\phi(z).$

\medskip

  Next, correspnding to the zeros of $\psi(z)-\phi(z)$ and $\psi'(z)-\phi(z),$ we have the following three cases:
	\begin{itemize}
	\item[Case 1:]  $z_0$ is a zero of $\psi(z)-\phi(z)$ but it is not a zero of $\psi'(z)-\phi(z).$ Then  $z_0$ is a pole of $G(z)$.
	\item[Case 2:]  $z_0$ is zero of $\psi'(z)-\phi(z)$ but it is not a zero of $\psi(z)-\phi(z).$ Then  $z_0$ is a zero  of $G(z)$.
	\item[Case 3:] $z_0$ is a common zero of $\psi'(z)-\phi(z) \mbox{ and }\psi(z)-\phi(z)$. Let $j$ and $k$ be the multiplicities of $z_0$ as a zero of $\psi'(z)-\phi(z) \mbox{ and }\psi(z)-\phi(z)$, respectively. Then near $z_0$,
		$$G(z)=t_2(z)\left(z-z_0\right)^{(l+1)j-(l+2)k}$$
		for some  function  $t_2(z)$ analytic in a neighborhood of $z_0 \mbox{ such that }t_2(z_0)\neq 0$.
		  
			\medskip
			
			 Thus $z_0$ is a pole of $G(z)$ if $k>\frac{l+1}{l+2}j$ and $z_0$ is a zero of $G(z)$ if $k<\frac{l+1}{l+2}j$.
\end{itemize}
   
	\medskip
	
	Let $N(r,\frac{1}{f},\frac{1}{g^0})$ be the counting function of zeros of $f$ which are not zeros of $g$ , $N(r,\frac{1}{f},\frac{1}{g})$ be the counting function corresponding to the common zeros of $f\mbox{ and }g$ and $N^{(\alpha)}(r,\frac{1}{f},\frac{1}{g})$ be the counting function corresponding to the common zeros of $f$ and $g$, such that the $m(f, z_0)>\alpha m(g,z_0)$, where by $m(f,z_0)$ we denote the multiplicity of $z_0$ as a zero of $f$. With these notations and   the preceding arguments , we find that

\begin{flalign}\label{eq:poles}
 \bar{N}\left(r,\frac{1}{G}\right) &\leq \bar{N}_2(r,f)+\bar{N}(r,\phi) + \bar{N}\left(r,\frac{1}{\psi'-\phi},\frac{1}{(\psi-\phi)^0}\right) \nonumber \\
&+\bar{N}^{\left(\frac{l+2}{l+1}\right)}\left(r,\frac{1}{\psi'-\phi},\frac{1}{\psi-\phi}\right) &&
\end{flalign}
and
\begin{flalign}\label{eq:zeros}
\bar{N}(r,G) &\leq \bar{N}\left(r,\frac{1}{\psi-\phi},\frac{1}{(\psi'-\phi)^0}\right)+\bar{N}^{\left(\frac{l+1}{l+2}\right)}\left(r,\frac{1}{\psi-\phi},\frac{1}{\psi'-\phi}\right).
\end{flalign}
Note that 
\begin{flalign*}
 &\bar{N}\left(r,\frac{1}{\psi-\phi},\frac{1}{(\psi'-\phi)^0}\right)+\bar{N}^{\left(\frac{l+1}{l+2}\right)}\left(r,\frac{1}{\psi-\phi},\frac{1}{\psi'-\phi}\right) &&\\
&+\bar{N}\left(r,\frac{1}{\psi'-\phi},\frac{1}{(\psi-\phi)^0}\right)+\bar{N}^{\left(\frac{l+2}{l+1}\right)}\left(r,\frac{1}{\psi'-\phi},\frac{1}{\psi-\phi}\right) && \\
&\leq \bar{N}\left(r,\frac{1}{\psi-\phi}\right)+\bar{N}\left(r,\frac{1}{\psi'-\phi}\right).&&
\end{flalign*}
Therefore, using (\ref{eq:poles}) and (\ref{eq:zeros}) in (\ref{eq:simple pole}), we get 
\begin{flalign}\label{eq:I1}
  lN_1\left(r,f\right) & \leq \bar{N}_2(r,f)+\bar{N}(r,\phi) + \bar{N}\left(r,\frac{1}{\psi-\phi}\right)
	+\bar{N}\left(r,\frac{1}{\psi'-\phi}\right)+m\left(r,\frac{G'}{G}\right)+O(1). &&
\end{flalign}
Since $T(r,\phi)=S(r,f)$ and $S(r,\psi)=S(r,f)$, by Theorem 3.1 in \cite{Hayman-1}, we have $$m\left(r,\frac{G'}{G}\right)=S(r,f).$$

  Thus from (\ref{eq:I1}), it follows that 
	\begin{flalign}
  lN_1\left(r,f\right) & \leq \bar{N}_2(r,f) + \bar{N}\left(r,\frac{1}{\psi-\phi}\right)
	+\bar{N}\left(r,\frac{1}{\psi'-\phi}\right)+S(r,f). &&
\end{flalign}
\qed

\begin{lemma}\label{22feb}
Let $f\in\mathcal{M}\left(\mathbb{C}\right)$  and let $\phi$ be a small function of $f$ such that $f \mbox{ and }\phi$ have no common poles.  Then
\begin{equation} lN_1\left(r,f\right)\leq \bar{N_2}\left(r,f\right)+\bar{N}\left(r,\frac{1}{\psi-\phi}\right)+ N_0\left(r,\frac{1}{\psi'-\phi'}\right)+S(r,f),
\end{equation}
 where $N_1\left(r,f\right)$ is the counting function of simple poles of $f$, $\bar{N_2}\left(r,f\right)$ is the counting function of multiple poles of $f$ counted once and  $N_0\left(r,\frac{1}{\psi'-\phi'}\right)$ is the counting function of zeros of $\psi'-\phi'$ which are not repeated zeros of $\psi-\phi$.
\end{lemma}
\begin{proof}
 Define $$G(z)={\left\{\psi'(z)-\phi'(z)\right\}^{l+1}\over {\left\{ \phi(z)-\psi(z)\right\}^{l+2}}}.$$ Then as in the proof of Lemma \ref{lemma3.1A} above,  we
 again arrive at (\ref{eq:simple pole}).
 
 Next, to find the distribution of poles and zeros of $G(z)$, we proceed as follows;
	 
	  Put, $$h(z)=\psi(z)-\phi(z).$$   	
     If $z_0$ is a pole of $h(z)$ of order $m,$ then near $z_0,$
$$h(z)={s(z)\over\left(z-z_0\right)^m}\mbox{  and  }h'(z)={t(z)\over\left(z-z_0\right)^{m+1}}~,$$
 where $s(z)\mbox{ and }t(z)$ are  functions analytic in a neighborhood of $z_0$ and both have no zeros at $z_0.$ So, 
\begin{eqnarray}\label{eq:6LA}
G(z)={w(z)\over\left(z-z_0\right)^{l+1-m}}
\end{eqnarray}
 for some  function  $w(z)$ analytic in a neighborhood of $z_0 \mbox{ such that }w(z_0)\neq 0.$ 

		Next if $z_0$ is a zero of $h(z),$ then near $z_0$, $h(z)={l(z)\left(z-z_0\right)^m}$ and so 
\begin{eqnarray}\label{eq:7LA}
G(z)={m(z)\over\left(z-z_0\right)^{l+1+m}}
\end{eqnarray} 
where $l(z)\mbox{ and }m(z)$ are functions analytic in a neighborhood of $z_0$ and both have no zeros at $z_0.$ 

    From (\ref{eq:7LA}) and (\ref{eq:6LA}) we see that the  only poles of $G(z)$ occur at
  \begin{itemize}
	\item[(i)] the roots of $h(z)=0$ and 
	\item[(ii)] the poles of $\phi(z)$ of multiplicity less than $l+1$.
	\end{itemize}
 Therefore,
\begin{eqnarray}\label{eq:8LA}
\bar{N}\left(r,G\right)\leq\bar{N}\left(r,\frac{1}{\psi-\phi}\right)+\bar{N}\left(r,\phi\right)-\bar{N}_{l+1}\left(r,\phi\right),
\end{eqnarray}
 where  $\bar{N}_{k}\left(r,\phi\right)$ is the counting function of  poles of $\phi(z)$  which have multiplicity atleast $k$, each pole is counted once.
   
	   Since the zeros of $G(z)$ occur at 
	\begin{itemize}
	\item[(i)] the roots of $h'(z)=0$ which are not the roots of $h(z)=0$
	\item[(ii)] multiple poles of $f(z)$ and 
	\item[(iii)] poles of $\phi$ of multiplicity greater than $l+1$,
\end{itemize}
 therefore, 
\begin{eqnarray}\label{eq:9LA}
\bar{N}\left(r,\frac{1}{G}\right)\leq \bar{N}_0\left(r,\frac{1}{\psi'-\phi'}\right)+\bar{N}_2\left(r,f\right)+\bar{N}_{l+1}\left(r,\phi\right).
\end{eqnarray}
Adding (\ref{eq:8LA}) and (\ref{eq:9LA}), we have;
\begin{flalign}\label{eq:10LA}
\bar{N}\left(r,G\right)+\bar{N}\left(r,\frac{1}{G}\right) \leq \bar{N_2}\left(r,f\right)+\bar{N}\left(r,\phi\right)
+\bar{N}\left(r,\frac{1}{\psi-\phi}\right)+N_0\left(r,\frac{1}{\psi'-\phi'}\right).
\end{flalign}
Since  $T(r,\phi)=S(r,f)$ and $S(r,\psi)=S(r,f)$, by Theorem $3.1$ in \cite{Hayman-1}, we have 
$$m\left(r,\frac{G'}{G}\right)=S(r,f) .$$ 

   Thus, from (\ref{eq:simple pole}) and (\ref{eq:10LA}), it follows that 
	\begin{equation}\label{eq:12LA}
	lN_1\left(r,f\right)\leq \bar{N_2}\left(r,f\right)+\bar{N}\left(r,\frac{1}{\psi-\phi}\right)+N_0\left(r,\frac{1}{\psi'-\phi'}\right)+ S(r,f).
	\end{equation}
\end{proof}
\begin{lemma}\label{final lemmaA} Let $f, \psi \mbox{ and }\phi$ be  as in Lemma \ref{lemma3.1A}. Then
\begin{flalign*}
 (a)~ T\left(r,f\right)&\leq \left(2+\frac{1}{l}\right)N\left(r,\frac{1}{f}\right)+\left(2+\frac{2}{l}\right)\bar{N}\left(r,\frac{1}{\psi-\phi}\right)\\
&+\frac{1}{l}\bar{N}\left(r,\frac{1}{\psi'-\phi}\right)-\left(2+\frac{1}{l}\right)N^0_2\left(r,\psi\right)+S(r,f).&&
\end{flalign*}
 \begin{flalign*}
(b)~   T\left(r,f\right)&\leq \left(2+\frac{1}{l}\right)N\left(r,\frac{1}{f}\right)+\left(2+\frac{2}{l}\right)\bar{N}\left(r,\frac{1}{\psi-\phi}\right)\\
&+\frac{1}{l}N_0\left(r,\frac{1}{\psi'-\phi'}\right)-\left(2+\frac{1}{l}\right)N^0_2\left(r,\psi\right)+S(r,f).&&
\end{flalign*}
\end{lemma}
\begin{proof}
By Lemma \ref{theorem 3.2}, we have
\begin{flalign}\label{eq:1FL}
 N_1\left(r,f\right)+2\bar{N_2}\left(r,f\right)&\leq N\left(r,f\right) \nonumber\\
&\leq T\left(r,f\right) &&\nonumber\\
&\leq\bar{N}\left(r,f\right)+N\left(r,\frac{1}{f}\right)+\bar{N}\left(r,\frac{1}{\psi-\phi}\right)-N^0_2\left(r,\psi\right)+S(r,f).&&
\end{flalign}
Since $\bar{N}\left(r,f\right)=N_1\left(r,f\right)+\bar{N_2}\left(r,f\right),$ from (\ref{eq:1FL}) we have  
\begin{eqnarray}\label{eq:2FL}
\bar{N_2}\left(r,f\right)\leq N\left(r,\frac{1}{f}\right)+\bar{N}\left(r,\frac{1}{\psi-\phi}\right)-N^0_2\left(r,\psi\right)+S(r,f) .
\end{eqnarray}

Using (\ref{eq:2FL}) in Lemma \ref{lemma3.1A}, we obtain 
\begin{flalign}\label{eq:3FL} 
N_1\left(r,f\right) &\leq \frac{1}{l} N\left(r,\frac{1}{f}\right)+\frac{2}{l}\bar{N}\left(r,\frac{1}{\psi-\phi}\right)-\frac{1}{l}N^0_2\left(r,\psi\right)\nonumber\\
&+\frac{1}{l}\bar{N}\left(r,\frac{1}{\psi'-\phi}\right)+ S(r,f).&&
\end{flalign}

Now from (\ref{eq:2FL}) and (\ref{eq:3FL}) it follows that  
\begin{flalign}\label{eq:4FL}
\bar{N}\left(r,f\right)&= N_1\left(r,f\right)+\bar{N_2}\left(r,f\right) \nonumber\\
&\leq N_1\left(r,f\right)+N\left(r,\frac{1}{f}\right)+\bar{N}\left(r,\frac{1}{\psi-\phi}\right)-N^0_2\left(r,\psi\right)+S(r,f) &&\nonumber\\
&\leq \left(1+\frac{1}{l} \right)N\left(r,\frac{1}{f}\right)+\left(1+\frac{2}{l}\right)\bar{N}\left(r,\frac{1}{\psi-\phi}\right)-\left(1+\frac{1}{l}\right)N^0_2\left(r,\psi\right) &&\nonumber\\
&+\frac{1}{l}\bar{N}\left(r,\phi\right)+\frac{1}{l}\bar{N}\left(r,\frac{1}{\psi'-\phi}\right)+S(r,f).&&
 \end{flalign}

Now in view of (\ref{eq:4FL}), Lemma \ref{theorem 3.2} yields 
\begin{flalign*}
T\left(r,f\right)&\leq \left(2+\frac{1}{l}\right)N\left(r,\frac{1}{f}\right)+\left(2+\frac{2}{l}\right)\bar{N}\left(r,\frac{1}{\psi-\phi}\right)\\
&+\frac{1}{l}\bar{N}\left(r,\frac{1}{\psi'-\phi}\right)-\left(2+\frac{1}{l}\right)N^0_2\left(r,\psi\right)+S(r,f),&&
\end{flalign*} 
which proves $(a)$.

 The conclusion $(b)$ follows by using Lemma \ref{22feb} instead of Lemma \ref{lemma3.1A} in the proof of $(a)$, above.
  
\end{proof}

 \textbf{Proof of Theorem \ref{MTA}: } Since $N^0_2\left(r,\psi\right)\geq 0$, by Lemma \ref{final lemmaA}$(a)$ we have
\begin{flalign}\label{eq:1}
T\left(r,f\right)&\leq 3N\left(r,\frac{1}{f}\right)+4\bar{N}\left(r,\frac{1}{\psi-\phi}\right)
 +\bar{N}\left(r,\frac{1}{\psi'-\phi}\right)+S(r,f). 
\end{flalign}

    Since $f\mbox{ and } \psi-\phi$ have no zeros,  $N\left(r,\frac{1}{f}\right)=0 \mbox{ and }\bar{N}\left(r,\frac{1}{\psi-\phi}\right)=0.$ Therefore, (\ref{eq:1})  reduces to 
\begin{eqnarray}\label{eq:4}
T\left(r,f\right)\leq \bar{N}\left(r,\frac{1}{\psi'-\phi}\right)+S(r,f). 
\end{eqnarray}
Since $f\in\mathcal{M}\left(\mathbb{C}\right)$ is transcendental, (\ref{eq:4}) implies that $\psi'(z)=\phi(z)$  has infinitely many solutions.\\
  Similarly, Lemma  \ref{final lemmaA}$(b)$ implies that $\psi'(z)=\phi'(z)$  has infinitely many solutions.   ~~~~~~~~~~~~~~~~~~~~~~~~~$\qed$


\bibliographystyle{amsplain}


\end{document}